\documentclass[12pt]{amsart} \font\emailfont=cmtt10

\usepackage{amsmath,amsthm,amsfonts,amscd,epsf,epsfig,verbatim,color,graphicx}
\usepackage{pinlabel}

\newtheorem{theorem}{Theorem}[section]
\newtheorem{lemma}[theorem]{Lemma}
\newtheorem{proposition}[theorem]{Proposition}
\newtheorem{corollary}[theorem]{Corollary}

\theoremstyle{definition}

\theoremstyle{remark} \newtheorem{remark}[theorem]{Remark}

\newcommand{\Q}{\mathbb{Q}} \newcommand{\Z}{\mathbb{Z}}

\DeclareMathOperator{\Kom}{Kom} \DeclareMathOperator{\Mat}{Mat}
\DeclareMathOperator{\Cob}{Cob} \DeclareMathOperator{\id}{id}

\title[{The Khovanov homology of 3-strand pretzels, revisited.}]  {The
  Khovanov homology of 3-strand pretzels, revisited.}

\author[Andrew Manion]{Andrew Manion} \address{Department of
  Mathematics, Princeton University, New Jersey 08544 \newline
  \indent{\emailfont{amanion@math.princeton.edu}}} \thanks{The author
  was supported by an NDSEG research fellowship.}

\begin{document}

\begin{abstract} We compute the reduced Khovanov homology of
$3$-stranded pretzel links. The coefficients are the integers with the
  ``even'' sign assignment. In particular, we show that the only
  homologically thin, non-quasi-alternating 3-stranded pretzels are
  $P(-p,p,r)$ with $p$ an odd integer and $r \geq p$ (these were shown
  to be homologically thin by Starkston \cite{Starkston} and Qazaqzeh
  \cite{Qazaqzeh}).
\end{abstract}

\maketitle
\section{Introduction} The purpose of this paper is to present a
computation of the reduced Khovanov homology (introduced in
\cite{KhovCat}) of all 3-strand pretzel links. There have been several
computations of Khovanov homology for partial families of pretzel
knots (see Starkston \cite{Starkston}, Suzuki \cite{Suzuki}, Qazaqzeh
\cite{Qazaqzeh}), and one by the author \cite {Manion} computing the
unreduced homology over $\Q$ for all 3-strand pretzels. Whereas
\cite{Manion} used the unoriented skein exact sequence in Khovanov
homology, this paper will use a shorter and more conceptual argument
relying on Bar-Natan's cobordism formulation of Khovanov homology for
tangles (see \cite{TangCob}). We will determine the reduced homology
over $\Z$, with the standard (``even'') sign assignment.

One caveat is required: the reduced Khovanov homology of
multi-component links depends on which component has the
basepoint. For $2$- and $3$-component 3-stranded pretzel links (those
$P(p,q,r)$ where two or three of $\{p, q, r\}$ are even), we will only
do the computation with one particular choice of basepoint. One could
apply the same method with the other basepoint choices, but the
details would be different enough that we decided not to write them
up.

\begin{figure} \centering
  \includegraphics{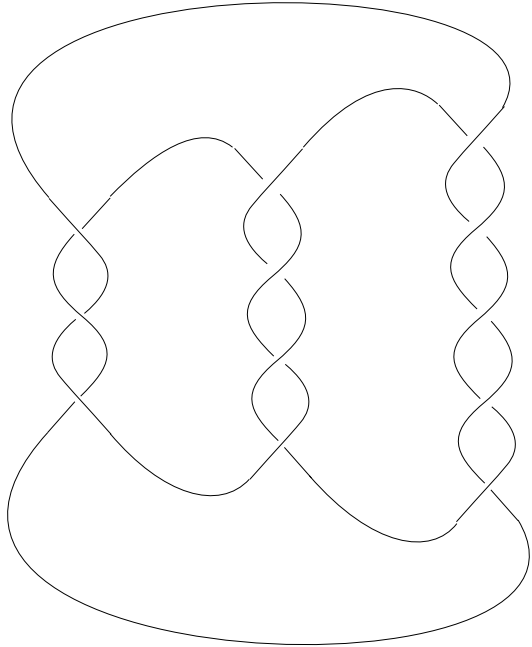}
  \caption{The pretzel knot $P(-3,4,5)$.}
  \label{neg1K}
\end{figure}

The usual diagram of the pretzel knot $P(-3,4,5)$ is shown in
Figure~\ref{neg1K}. We will use the equivalent diagram in
Figure~\ref{3K}. The general 3-strand preztel link is $P(p,q,r)$ where
$p$, $q$, and $r$ are arbitrary integers. Up to mirroring, though, we
may assume that at most one of $\{p, q, r\}$ is negative. If none are
negative, then the link is alternating and its Khovanov homology is
determined by its signature and Jones polynomial, by a result of Lee
from \cite{Lee}. So we will restrict attention to $P(-p,q,r)$ with
$\{p,q,r\}$ positive. We may even assume $q \leq r$ for convenience,
because of the symmetry of 3-stranded pretzels.

For some choices of $p$, $q$, and $r$, $P(-p,q,r)$ is
quasi-alternating, and the same argument applies as with alternating
links. Hence we may restrict attention to non-quasi-alternating
pretzel links. Results of Champanerkar-Kofman in \cite{ChKofQA} and
Greene in \cite{Greene} imply that these are $P(-p,q,r)$ with $p,q,r$
all positive, $p \leq q, r$ (or $p \leq q \leq r$), and $p \geq
2$. (Note that a simple isotopy of the standard diagram for
$P(-1,q,r)$ yields an alternating diagram.)

The formula for the bigraded Khovanov homology of these links takes a
bit of work to write down. Complicating matters is the issue of
orientations; to get the right gradings, we must decide on the
relative orientations of components of $P(-p,q,r)$ when it is a
multi-component link. To avoid cluttering the introduction, we will
state a simpler formula here. Recall the $\delta$-grading on Khovanov
homology; we will define it as $\delta = q/2 - h$, where $q$ is the
quantum grading and $h$ is the homological grading. (In \cite{KPKH},
Rasmussen defines it as $q - 2h$; the $1/2$ is a matter of
preference.) With this collapse of the gradings, Khovanov homology
becomes a singly-graded theory.

The link $P(-p,q,r)$ is a knot when at most one of $\{p, q, r\}$ are
even. We need not consider the case of ``only $r$ even'' separately
from the case of ``only $q$ even,'' since these cases are interchanged
by the symmetry of pretzel knots (here we're not requiring $q \leq
r$). Below we state the $\delta$-graded formula for knots; this has
the advantage that we do not need to mention orientations at all.

\begin{theorem}\label{maintheorem} Let $p,q,r$ be as above, such that
$P(-p,q,r)$ is a non-quasi-alternating knot (for links, see
  Theorem~\ref{maintheorem2}). Let $H_{\delta}$ be the reduced
  Khovanov homology of the knot $P(-p,q,r)$ in grading $\delta$.
\begin{itemize}
\item If $p,q,r$ are all odd, then $H_0 = \Z^{p^2 - 1}$ and $H_{-1} =
  \Z^{(q-p)(r-p) - 1}$. All other $H_{\delta}$ are 0.
\item If $p$ is even, then $H_{\frac{q+r}{2}} = \Z^{p^2}$ and
  $H_{\frac{q+r}{2} - 1} = \Z^{(q-p)(r-p)}$. All other $H_{\delta}$
  are 0.
\item If $q$ is even, then $H_{\frac{-p+r}{2}} = \Z^{p^2 - 1}$ and
  $H_{\frac{-p+r}{2} - 1} = \Z^{(q-p)(r-p) - 1}$ in $\delta =
  \frac{-p+r}{2} - 1$. All other $H_{\delta}$ are 0.
\end{itemize} When $p$ is odd and $p = q$, the formula gives a $-1$ in
the lower $\delta$-grading. This should be interpreted as a $0$, with
a $1$ added to the rank of the higher $\delta$-grading.
\end{theorem}

In the course of proving this theorem (or rather
Theorem~\ref{maintheorem2} for links), we will see how the bigraded
homology could be computed without any more real work. In fact, we
will ignore gradings throughout most of the paper, and then deduce
them at the end when needed.

\subsection{Homological thinness.}  Starkston in \cite{Starkston} and
Qazaqzeh in \cite{Qazaqzeh} were interested in the class of
homologically thin, non-quasi-alternating pretzel knots. Starkston
conjectured, and Qazaqzeh proved, that all $P(-p,p,r)$ pretzel knots
with $p$ odd and $r \geq p$ are homologically thin but not
quasi-alternating. Our results here (e.g. Theorem~\ref{maintheorem2})
imply that these are the only homologically thin,
non-quasi-alternating 3-column pretzel links. All other non-QA ones
(including, e.g. $P(-p,p,r)$ with $p$ even and $r \geq p$) are
homologically thick. The same result could be deduced from the
unreduced homology calculated in \cite{Manion}, but this paper's
emphasis on the $\delta$-grading makes it easier to see.

\subsection{Acknowledgements.}  The author would like to thank
Zolt{\'a}n Szab{\'o} for many helpful discussions during the course of
this work.

\section{Khovanov homology computation.}  We will use Bar-Natan's
dotted cobordism formulation of Khovanov homology in this section; see
Section 11.2 of \cite{TangCob}. In particular, if $T$ is a tangle
diagram, then its Khovanov chain complex is an object of the category
$\Kom(\Mat(\Cob^3_{\cdot / l}))$.

\subsection{Local preliminaries.}  The computation will use a lemma
about the formal dotted-cobordism complex associated to a series of
$n$ half-twists. Effectively, the lemma is Proposition 25 of
Khovanov's original paper \cite{KhovCat}, interpreted in the language
of dotted cobordisms. We will give a proof here, to keep this paper as
self-contained as possible.

First, though, we recall a fact about dotted cobordism pictures. Let
$D$ be a crossingless tangle diagram. If $D$ contains a complete
circle $c$, then $D$ is isomorphic in the category $\Mat(\Cob^3_{\cdot
  / l})$ to $D' \oplus D'$, where $D'$ is $D$ with $c$ removed. This
``delooping'' isomorphism is written down by Bar-Natan in
\cite{FastKH}; the proof consists of a diagram, reproduced here for
convenience (with a few modifications) in Figure~\ref{0K}.

Whenever we apply a delooping isomorphism to remove a circle, one of
the two resulting summands will have a dot in the lower-right corner
and one will not. The dot indicates the element whose $q$-grading is
shifted by $-1$ rather than $+1$. The maps $F$ and $G$ in
Figure~\ref{0K} are inverses of each other, proving the delooping
isomorphism.

\begin{figure}
\labellist \small \hair 2pt \pinlabel $\emptyset$ at 176 127 \pinlabel
$\emptyset$ at 176 35 \pinlabel $\oplus$ at 176 80 \pinlabel $F$ at 56
2 \pinlabel $G$ at 290 2 \endlabellist \centering \includegraphics{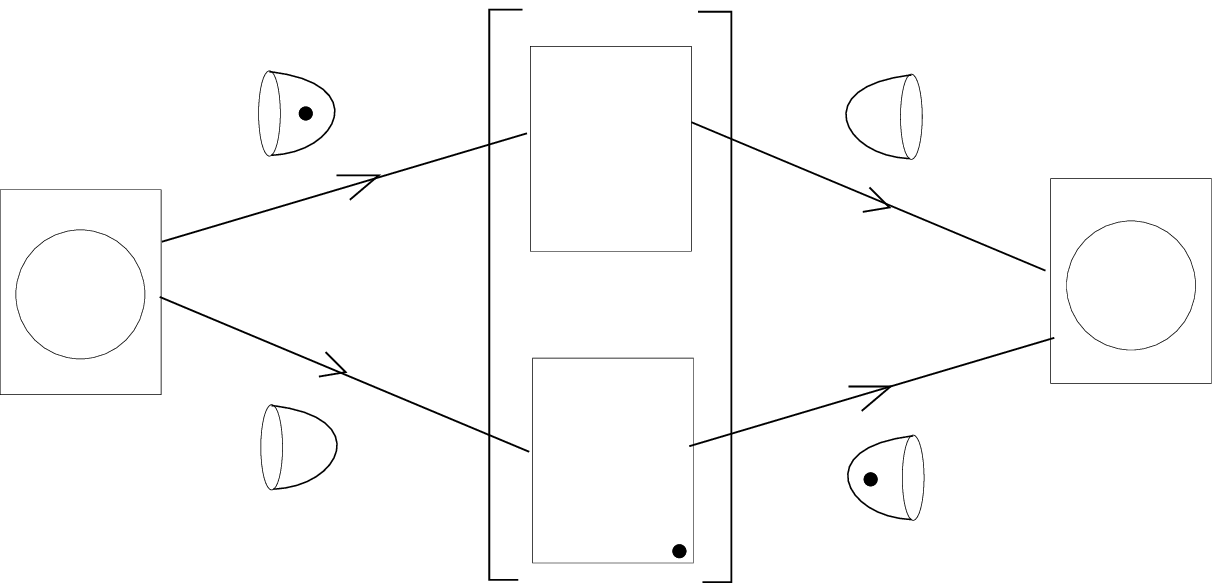}
  \caption{Bar-Natan's delooping isomorphism, taken from
    \cite{FastKH}.}
  \label{0K}
\end{figure}

The lemma we need is the following:

\begin{lemma}\label{inductive} The (formal) Khovanov chain complex of
the positive $n$-half-twisted strand on the left side of
Figure~\ref{1K} is homotopy equivalent to the dotted-cobordism complex
on the right side of Figure~\ref{1K}.
\end{lemma}

\begin{figure}
\labellist \small \hair 2pt \pinlabel $n$ at 49 56 \pinlabel $n$ at
286 6 \pinlabel $[00\ldots0]$ at 135 81 \pinlabel $[10\ldots0]$ at 193
81 \endlabellist \centering \includegraphics{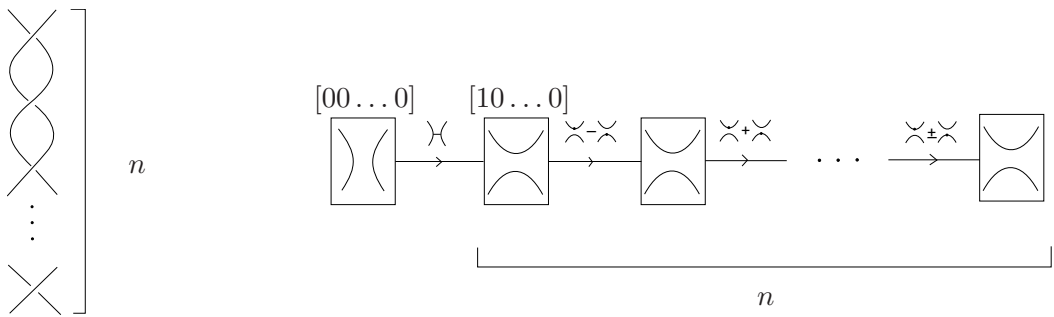}
  \caption{Simplification of Kh(positive $n$-twisted tangle). There
    are $n$ boxes (and $n-1$ maps) in the section labeled on the
    right.  The final sign is $-$ for $n$ even and $+$ for $n$ odd.}
  \label{1K}
\end{figure}

\begin{remark} We have ignored gradings in Lemma~\ref{inductive}
because it's not necessary for our purposes to keep track of them
now. We will be able to deduce them later given our knowledge of the
boundary map. However, due to signs, it is still important to order of
the crossings. We order the crossings in the $n$-half-twisted strand
from bottom to top, and on the right side of the diagram, we show how
the first few generators are labelled.
\end{remark}

\begin{proof}[Proof of Lemma~\ref{inductive}.]  We will induct on $n$;
the case $n = 1$ is true by the definition of the formal complex for a
single crossing. Assume the lemma is true for $n - 1$; then we can use
the induction hypothesis to replace the formal complex for the
$n$-twisted strand by the one shown at the top of
Figure~\ref{2K}. Delooping to get rid of the complete circles, we get
the complex at the bottom of Figure~\ref{2K}.

\begin{figure}
\labellist \small \hair 2pt \pinlabel $a_1$ at 392 83 \pinlabel $a_2$
at 314 83 \pinlabel $a_3$ at 242 83 \pinlabel $a_{n-1}$ at 74 83
\pinlabel $-\id$ at 44 105 \pinlabel $-\id$ at 145 131 \pinlabel
$-\id$ at 230 131 \pinlabel $-\id$ at 308 131 \pinlabel $-\id$ at 386
131 \pinlabel $n-1$ at 259 210 \pinlabel $n-1$ at 255 -9 \endlabellist
\centering \includegraphics{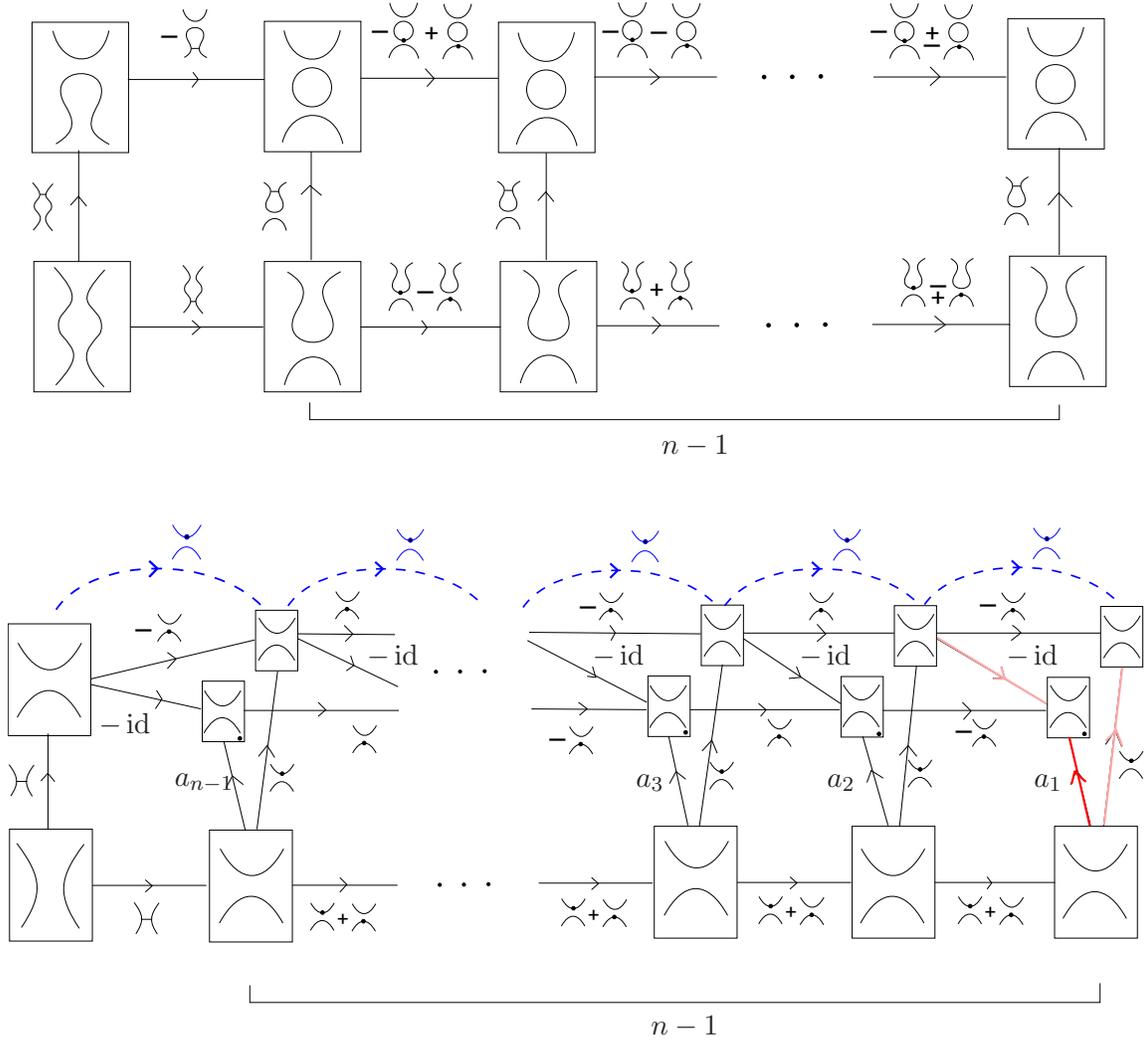}
  \caption{Inductive step of Lemma~\ref{inductive}.}
  \label{2K}
\end{figure}

Now we simplify using Gaussian elimination, as described in this
context by Bar-Natan in \cite{FastKH}. Whenever we see an invertible
matrix coefficient in the differential, we remove the two
corresponding generators, and add in some ``zig-zag'' terms. Suppose
the invertible coefficient is $a$, from generator $g_1$ to generator
$g_2$ (write $a: g_1 \rightarrow g_2$ for convenience). Whenever we
have $b: h \rightarrow g_2$ and $c: g_1 \rightarrow k$, we must add
$-c \circ a^{-1} \circ b$ to the coefficient from $h$ to $k$. Then the
resulting complex with $g_1$ and $g_2$ removed is homotopy equivalent
to the original one.

The edges labelled $a_1, \ldots, a_{n-1}$ all represent identity
cobordisms with positive sign; we eliminate them in order, starting
with $a_1$. Each elimination has one negative sign from the edges and
another from the elimination formula, so the dotted maps in
Figure~\ref{2K} get positive signs. The result is the complex we want.
\end{proof}

\begin{figure}
\labellist \small \hair 2pt \pinlabel $n$ at 285 10 \pinlabel $n$ at
45 63 \pinlabel $[11\ldots1]$ at 135 81 \pinlabel $[01\ldots1]$ at 193
81 \endlabellist \centering \includegraphics{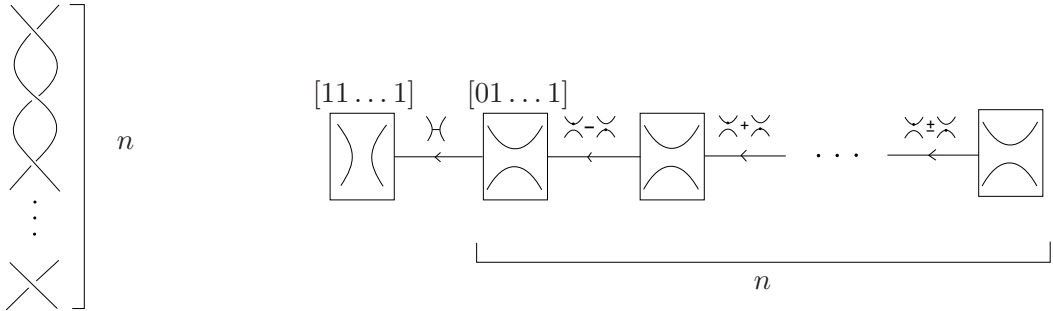}
  \caption{The negative analogue of Figure~\ref{1K}.}
  \label{31K}
\end{figure}

Taking the mirror image of this lemma gives us:

\begin{corollary}\label{mirror} The (formal) Khovanov chain complex of
the negative $n$-half-twisted strand on the left side of
Figure~\ref{31K} is homotopy equivalent to the dotted-cobordism
complex on the right side of Figure~\ref{31K}.
\end{corollary}

Strictly speaking, the global sign on any edge could turn out to be
the opposite of the one depicted in Figure~\ref{31K}. But, up to
isomorphism of complexes, we may assume the signs are as shown. Note
that if we labelled the crossings in the $n$-half-twisted strand ``up
to down'' rather than ``down to up,'' we would get some complex that
looks identical except for global signs on the edges (the relative
signs on each edge are required by $d^2 = 0$). So by the same logic,
the result is independent of this choice up to isomorphism.

\begin{figure}
\labellist \small \hair 2pt \pinlabel $p$ at 109 140 \pinlabel $q$ at
48 89 \pinlabel $r$ at 151 88 \pinlabel $1$ at 182 163 \pinlabel $2$
at 22 174 \pinlabel $3$ at 94 16 \endlabellist \centering
\includegraphics{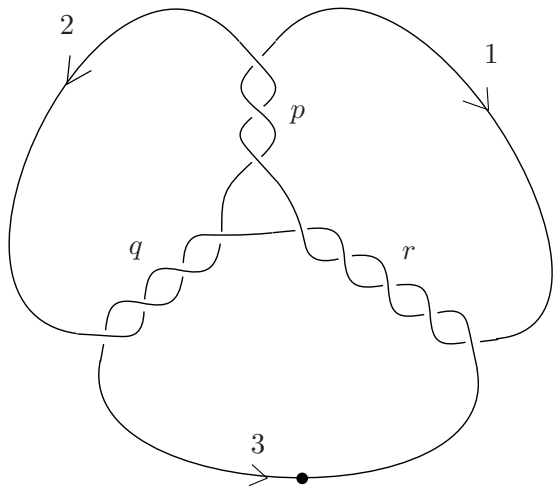}
  \caption{The pretzel knot $P(-3,4,5)$ (or $(-p,q,r)$ in general).}
  \label{3K}
\end{figure}

\subsection{Diagrams and orientations.}  Now consider the pretzel link
$P(-p,q,r)$, with a diagram $D$ and basepoint drawn as in
Figure~\ref{3K}. Order the crossings so that those in the $p$-strand
come first, then those in the $q$-strand, then those in the
$r$-strand. Within each strand, the crossings should be ordered from
one end to the other; as noted above, it doesn't matter which end is
which.

Since the Khovanov homology of links depends on relative orientations
between the components, we need to be able to specify these
orientations. We will do this here, although it will not be needed
until we fix absolute gradings at the end.

There are three arrows in Figure~\ref{3K}, labelled $1$, $2$, and $3$
(the directions are chosen somewhat arbitrarily to agree with the
example in that diagram). We will say a pretzel link has orientation
$+++$ if its orientations agree with the three arrows, or $-++$ if
they disagree at position 1 but agree at 2 and 3, etc. Since Khovanov
homology is invariant under overall change of orientation, we may fix
once and for all a $+$ in the third column, i.e. our links will always
be oriented in agreement with arrow 3. Then we may simply write $++$,
$-+$, etc. for the orientation of the link at positions 1 and 2.

When $P(-p,q,r)$ is a knot, the orientations at positions 1 and 2 are
fixed. If $q$ is even (and $p$ and $r$ are odd), like in
Figure~\ref{3K}, the orientation is $++$. If only $p$ is even, the
orientation is $+-$. If only $r$ is even, the orientation is
$--$. Finally, if $p$, $q$, and $r$ are all odd, the orientation is
$-+$.

When $P(-p,q,r)$ is a two-component link, two out of the four
orientation patterns are possible (depending on the parity of $p$,
$q$, and $r$). When it is a three-component link, all four patterns
are possible.

If we know the orientation pattern of the link, we can deduce the
number of positive and negative crossings; these will arise in the
grading formulas. We summarize them in Table~\ref{nplusnminus} for
convenience.

\begin{table}
  \begin{center}
    \begin{tabular}{|c||c|c|} \hline Pattern & $n_+$ & $n_-$ \\
\hline\hline $++$ & $r$ & $p+q$ \\ \hline $+-$ & $p+q+r$ & $0$
\\ \hline $-+$ & $p$ & $q+r$ \\ \hline $--$ & $q$ & $p+r$ \\ \hline
    \end{tabular}
  \end{center}
  \caption{Values for $n_+$, the number of positive crossings in the
    diagram, and $n_-$, the number of negative crossings, given the
    orientation pattern.}
  \label{nplusnminus}
\end{table}

\subsection{Computation of the complex.}  We now analyze the
dotted-cobordism Khovanov complex of $D$. The tangle complex of each
of the three twisted strands of $D$ may be simplified using
Lemma~\ref{inductive}. The formal complex of $D$ is then homotopy
equivalent to the cube shown in Figure~\ref{4K}, with
$(p+1)(q+1)(r+1)$ generators. Our goal will be to simplify this cube
even further, using delooping and elimination, until we understand the
differential completely.

While the resolution diagrams in Figure~\ref{4K} and subsequent
figures omit the basepoint for convenience, there should always be a
basepoint on each diagram as indicated in Figure~\ref{3K}. This
basepoint causes many cube differentials to be zero; those marked zero
in Figure~\ref{4K} are zero because they have a dot on the basepoint
component.

\begin{figure}
\labellist \small \hair 2pt \pinlabel $p+1$ at 468 163 \pinlabel $q+1$
at 69 33 \pinlabel $r+1$ at 359 28 \pinlabel $=0$ at 392 288 \pinlabel
$=0$ at 449 262 \pinlabel $=0$ at 14 124 \pinlabel
$[1\ldots1,0\ldots0,0\ldots0]$ at 213 355 \pinlabel
$[1\ldots1,0\ldots0,1\ldots1]$ at 62 308 \endlabellist \centering
\includegraphics{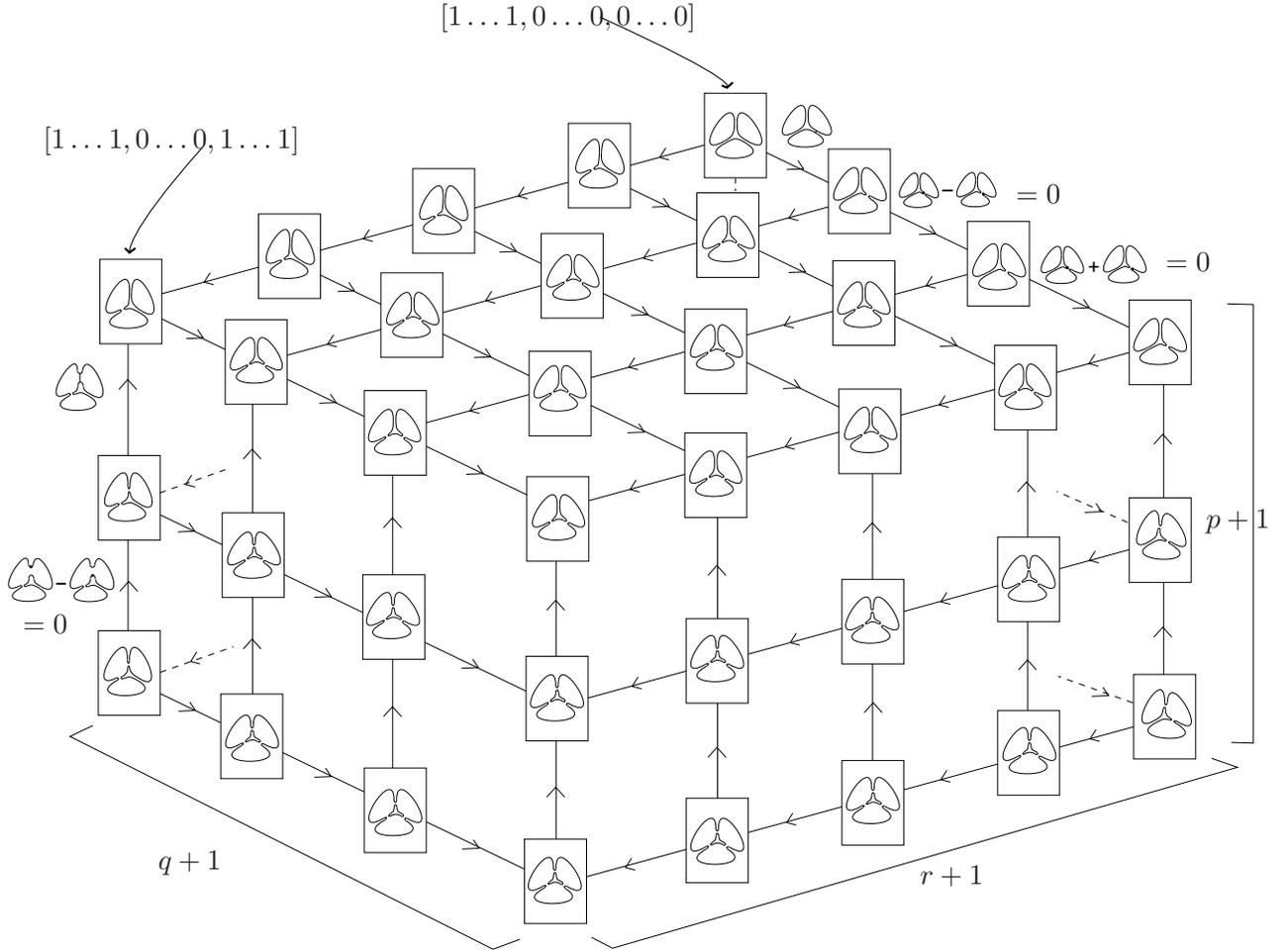}
  \caption{Chain complex homotopy equivalent to $Kh(P(-p,q,r))$, where
    $(p,q,r)$ = $(-2,3,4)$ in the picture. All of the arrows should
    have labels; only some are shown for convenience. We have also
    shown the crossings-label of two generators to indicate the
    general pattern. The generator $[0\ldots0,0\ldots0,0\ldots0]$ is
    precisely the corner of the cube obscured to the viewer.}
  \label{4K}
\end{figure}

\subsubsection{Columns in the cube.}  We start by thinking of the cube
as made up of vertical columns. Consider the columns which are not on
either of the two ``back walls'' of the cube as drawn in
Figure~\ref{4K}. We will consider these columns one at a time,
starting with the one closest to the viewer in Figure~\ref{4K}.

The left side of Figure~\ref{5K} shows any of these columns. According
to the discussion above, we can choose (once and for all) to make the
signs on the $p$-columns as depicted. (This corresponds to choosing
different global signs on various edges in Corollary~\ref{mirror}.) We
then apply delooping to get the column on the right of
Figure~\ref{5K}.
 
\begin{figure}
\labellist \small \hair 2pt \pinlabel $0$ at 47 250 \pinlabel $0$ at
47 197 \pinlabel $0$ at 47 129 \pinlabel $0$ at 47 73 \pinlabel $-\id$
at 188 292 \pinlabel $-\id$ at 202 242 \pinlabel $-\id$ at 200 183
\pinlabel $-\id$ at 206 118 \pinlabel $-\id$ at 206 65 \endlabellist
\centering \includegraphics{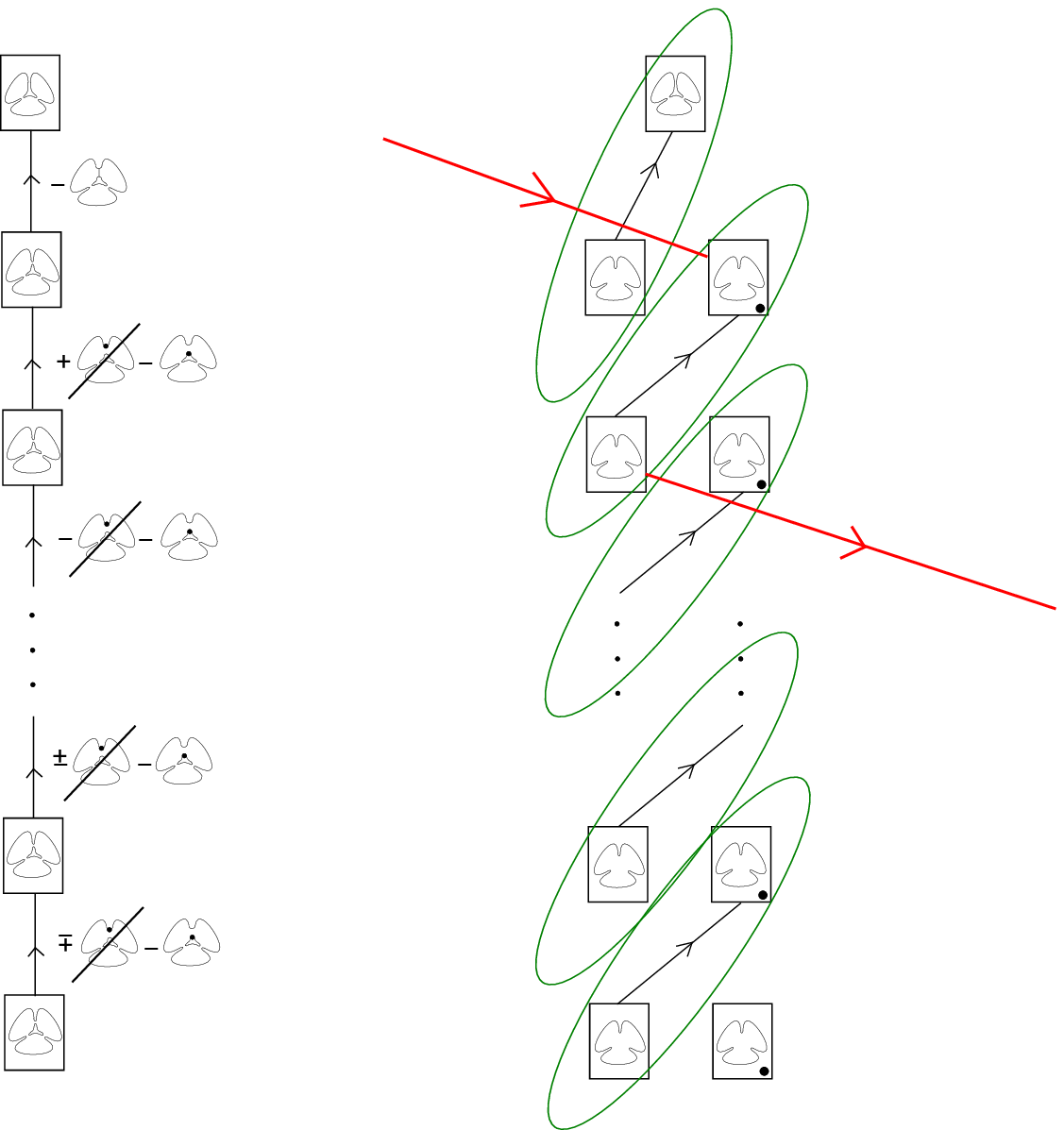}
  \caption{A vertical column of the cube, not on the back walls. The
    red arrows on the right are meant to suggest the zig-zag maps
    which result from cancelling circled pairs of generators.}
  \label{5K}
\end{figure}

\subsubsection{Paths.}  We want to see what happens when we simplify
each (non-back-wall) column by applying Gaussian elimination to the
maps marked $-\id$ in Figure~\ref{5K}. We will get a complex like the
one in Figure~\ref{6K}. It has two ``walls'' and a ``floor''; note
that we have applied delooping to the top vertices on each wall. In
this diagram, when a circle is delooped, we depict it as a dashed
circle in each resulting generator. There is a dot on one of the two
dashed circles.

The arrows in Figure~\ref{6K} all represent components of the
differential. There are more differentials like the red ones, based at
each wall generator except for the column where the walls
intersect. These arrows all point at floor generators, and they are
due to zig-zag maps arising from the eliminations we did in the
columns.

\begin{figure}
\labellist \small \hair 2pt \pinlabel $(=0)$ at 99 187 \endlabellist
\centering \includegraphics{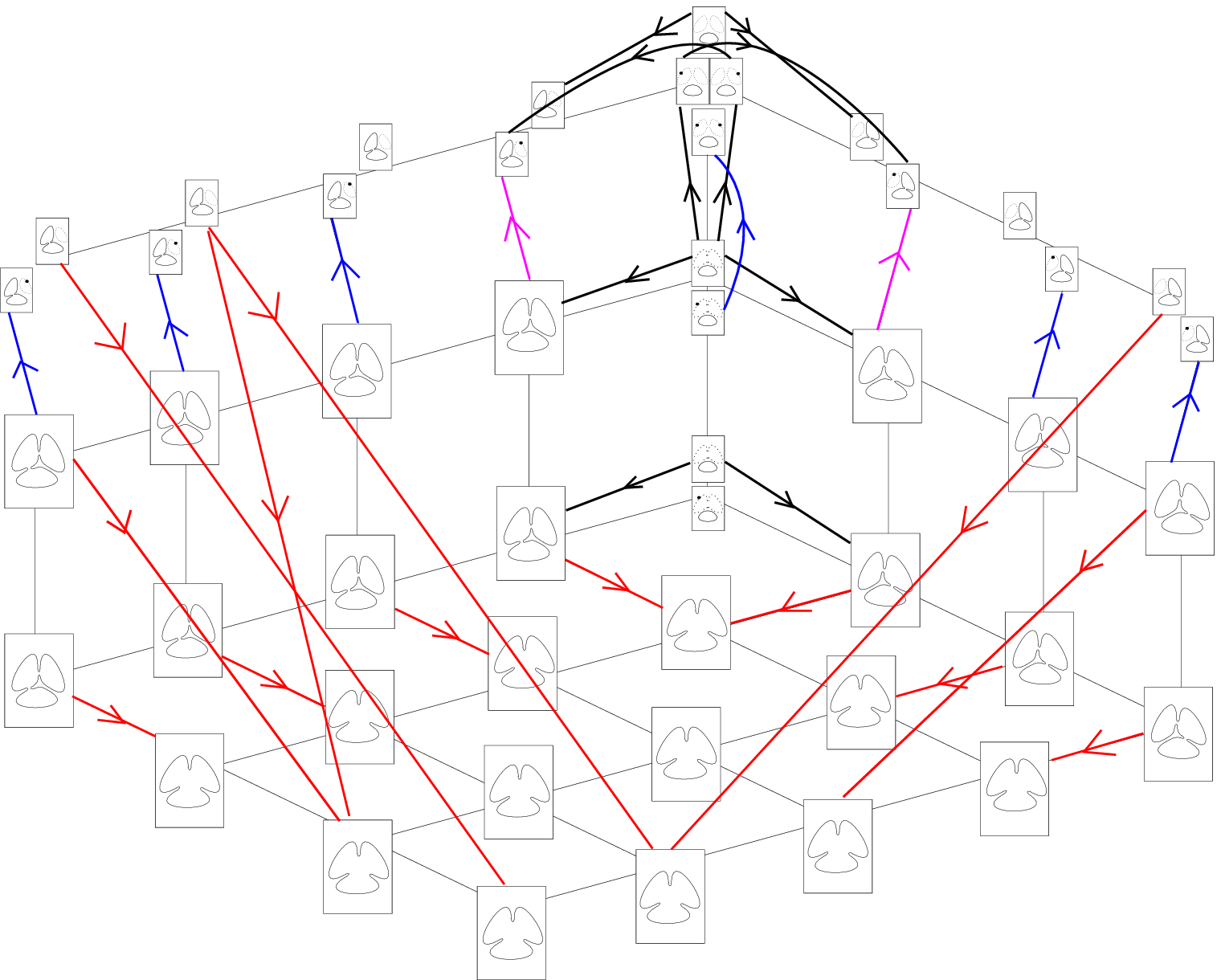}
  \caption{The cube after doing eliminations of all non-back-wall
    columns. All arrows represent components of the differential;
    there are more differentials like the red arrows that are not
    shown.}
  \label{6K}
\end{figure}

Such a component, arising from repeated eliminations, looks like a
path of arrows. The colored paths on the left of Figure~\ref{61K} are
examples. The paths must start and end in the positive $x$- or
$y$-direction, and must alternate such ``horizontal'' steps (always
positive) with vertical steps down (via the eliminated identity
components in the interior columns). Paths starting on the left wall
must start in the $y$-direction, and paths starting on the right wall
must start in the $x$-direction.

Each horizontal step in the path is $\pm \id$, except for steps
outward from the very tops of the walls. These top steps are join maps
before delooping; after, the ``top'' component is $\pm \id$ and the
``bottom'' one is $\pm$(dotted $\id$). However, none of the
dotted-$\id$ maps can contribute to nonzero maps remaining after
elimination, because they (eventually) put a dot on the
basepoint. Hence the only nonzero red arrows come from paths with each
step $\pm \id$.

Let $w$ be a (remaining) wall generator and $f$ be a floor
generator. Then $dw$ contains $n$ copies of $f$ where $n$ is the
signed count of paths from $w$ to $f$.

\subsubsection{Signs on paths.}  To actually compute $n$, we need to
pin down the signs on the paths. It is easiest to do this with an
example. In Figure~\ref{61K}, there is one path connecting $w$ to $f$
and three paths connecting $w'$ to $f'$. These paths are marked with
various colors. There is also a pattern drawn in orange at the right
of Figure~\ref{61K}. To compute the sign of a path, project it down to
the floor of the cube and walk along it backwards (starting from the
floor generator). Start with a $+$ sign. On the first step (and all
odd-numbered steps), the sign flips if the path traverses a black
edge. On the second step (and all even-numbered steps), the sign flips
if the path traverses an orange edge. The resulting sign is the sign
of the path. (Equivalently, we draw the opposite of the orange pattern
every other level, and then always flip on black edges.)

\begin{figure}
\labellist \small \hair 2pt 
\pinlabel {positive $x$} at 205 190
\pinlabel {positive $y$} at 318 178 
\pinlabel $w$ at 8 139 
\pinlabel $f$ at 120 -4 
\pinlabel $f'$ at 173 15 
\pinlabel $w'$ at 202 159
\endlabellist  
\centering
\includegraphics{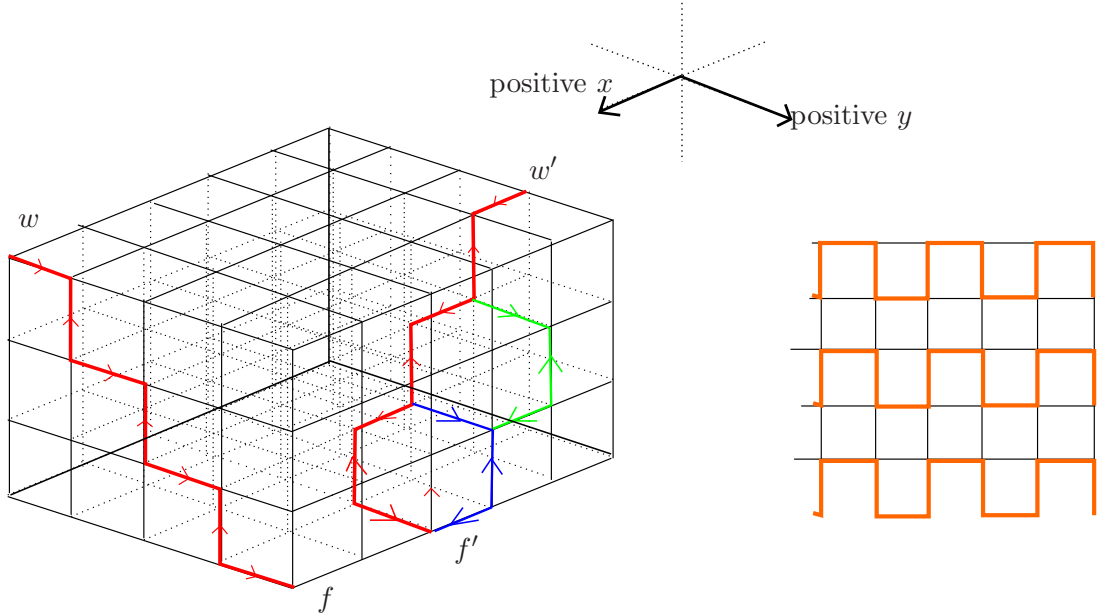}
  \caption{On the left: paths which contribute to the differential. In
    this diagram, $(p,q,r) = (3,4,5)$.  On the right: pattern of
    positive signs (orange) and negative signs (black) on the bottom
    layer of the cube.}
  \label{61K}
\end{figure}

For instance, the path from $f$ back up to $w$ goes B, O, B, O (where
B denotes a black edge and O denotes an orange edge), so its sign is
$+$ after four flips. The red path from $f'$ back up to $w'$ goes B,
B, O, B, so it gets $-$ after one flip. The blue path (O B O B) gets
$+$ (zero flips), and the green path (O B B B) gets $-$ (one flip).

To see why these signs are correct, note that the orange pattern is
just the standard pattern of positive signs in a double complex. In
the full triple complex we are considering, the vertical arrows are
all negative and the pattern of positive horizontal signs switches
every vertical layer.  But every time we eliminate a vertical arrow in
a column simplification, the negative sign on the arrow cancels the
negative sign from the Gaussian elimination. Thus the sign on the
resulting arrow is just the product of the signs on its horizontal
components, and this pattern continues to hold even after many
cancellations. Hence the signs on the paths are as described.

\begin{remark}\label{cubediffer} Note for future purposes that when
two paths ``differ by a cube'', as the red/blue paths or the
blue/green paths are related, then they have the opposite sign. To see
why this is true, color everything with an orange/black pattern which
flips every level, so sign flips are always black edges (as mentioned
above). Consider (e.g.) the cube spanned by the red and blue
paths. and the squares on the top and bottom of this cube. These
squares each have a ``source'' vertex with two outgoing edges and a
``sink'' vertex with two incoming edges. Since the pattern is such
that every square has an odd number of orange/black edges, the
outgoing edges match color precisely when the incoming edges
don't. Furthermore, the outgoing edges match color on the top iff the
outgoing edges match on the bottom, and the same is true for incoming
edges.

Now, the difference between the red and blue paths has two outgoing
edges, two incoming edges, and two vertical edges which are
positive. If the outgoing edges match, the incoming edges don't, so
the total sign difference is $-1$, and similarly if the incoming edges
match. Thus any paths differing like the red and blue paths must
differ in sign.
\end{remark}

\subsubsection{Cancelling more differentials.}  Figure~\ref{7K} shows
the complex of Figure~\ref{6K} from the top and rotated a bit
clockwise. For now we will assume $p + 2 \leq q$, $r$; this is the
generic case. Later we can look back and see what happens when $q = p$
or $q = p+1$ (without loss of generality we may assume $q \leq r$).

No paths are long enough to hit the floor generators circled with a 1
in Figure~\ref{7K}, so these generators survive to homology. There are
$(q-p-1)(r-p-1)$ of them. We want to determine which other floor
generators survive to homology.

\begin{figure}
\labellist \small \hair 2pt \pinlabel (1) at 25 89 \pinlabel (2) at 25
136 \pinlabel (3) at 179 30 \pinlabel (4) at 174 106 \pinlabel (5) at
107 133 \pinlabel (6) at 153 58 \pinlabel $f_1$ at 131 86 \pinlabel
$f_2$ at 171 131 \pinlabel $w_1$ at 138 205 \pinlabel $w_3$ at 162 205
\pinlabel $w_5$ at 188 205 \pinlabel $w_6$ at 253 143 \pinlabel $w_4$
at 253 118 \pinlabel $w_2$ at 253 93 \pinlabel wall at 125 220
\pinlabel wall at 295 106 \pinlabel $r+1$ at 125 265 \pinlabel $q+1$
at 372 106 \endlabellist \centering \includegraphics{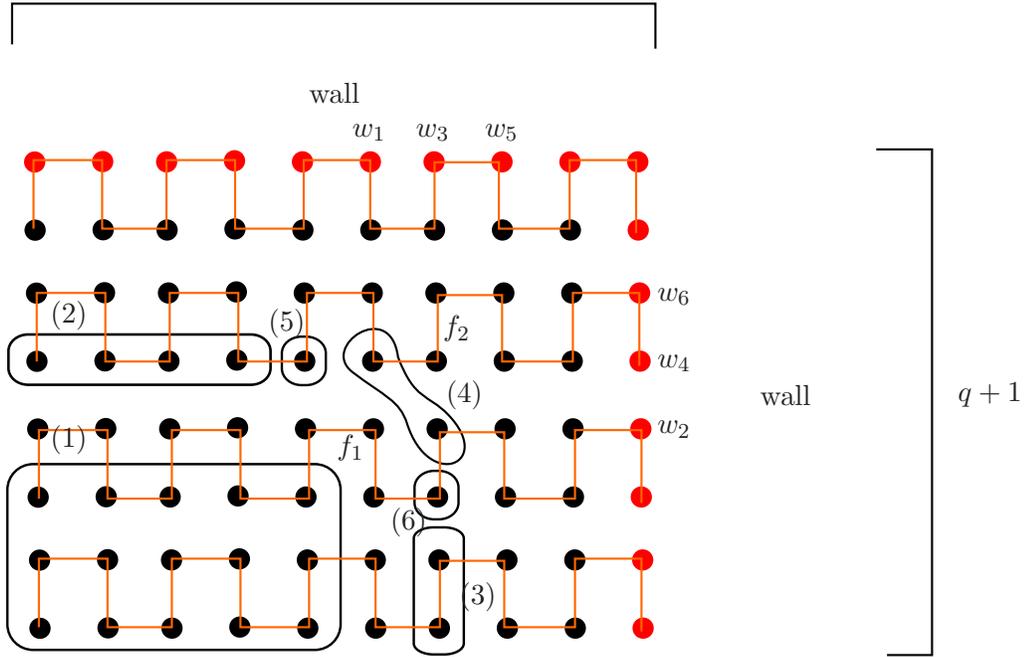}
  \caption{Top-down view of the remaining generators, for $(p,q,r) =
    (3,7,9)$. The pattern of positive signs on the bottom layer is
    also shown (in orange) for convenience.}
  \label{7K}
\end{figure}

First, look at the blue differentials in Figure~\ref{6K}. No other
arrows have the same tip as these, so they may all be cancelled
without picking up extra zig-zag maps. The purple arrows are almost
like the blue ones, except for the presence of black arrows with the
same tip as the purple ones. If we wanted, we could still cancel these
and pick up some extra maps, but instead we will leave them be for
now; later we will be able to cancel them without extra maps.

After cancelling the blue arrows, nearly all of the second-to-top-row
wall generators are gone. Now consider the uncircled dots in the
leftmost column of Figure~\ref{7K}. Each gets matched up with a wall
generator in the leftmost wall column, via arrows like the three red
arrows on the far left of Figure~\ref{6K}. Each arrow counts only one
path, so it is automatically $\pm \id$. The leftmost dot lying in the
circle ``2'' would get matched with the missing generator from the
second-to-top row. Cancelling arrows, we see that all the uncircled
dots die but the circled dot survives to homology. No extra maps are
generated because no other arrows share a tail with any of the
cancelled arrows.

Having dealt with the leftmost column, we move inward and do the same
thing with the next one. All the dots marked 2 (resp. 3) survive to
homology, and uncircled dots in their columns (resp. rows) do not.

With the dots marked 5 and 6, all we can say right now is that they
are potentially homology generators. To see why, look at the dot
marked $f_1$. This dot is the first to get hit by two wall generators
$w_1$ (from the north wall) and $w_2$ (from the east wall), i.e. it is
the first place where our cancellation process runs into
complications.

To determine what happens here, we want to identify the maps from
$w_1$ to (5) and from $w_2$ to (6). The parity of $p$ becomes
important here:
\begin{proposition}\label{mapprop} If $p$ is even, the maps $w_1
\rightarrow$ (5) and $w_2 \rightarrow$ (6) are zero. If $p$ is odd,
these maps are $\pm \id$.
\end{proposition}
\begin{proof} Without loss of generality, consider $w_1 \rightarrow$
(5). There are $p$ paths contributing to this matrix coefficient, and
  they are arranged in a sequence like the red/blue/green paths in
  Figure~\ref{61K}. As noted in Remark~\ref{cubediffer}, since each
  path in the sequence differs from the next by a cube, the signs
  alternate along the sequence. Hence the total sum is $0$ if $p$ is
  even and $\pm 1$ if $p$ is odd.
\end{proof}

If $p$ is even, we can choose (arbitrarily) to cancel the arrow from
$w_1$ to $f_1$, and we pick up no extra maps from $w_2$ to (5) or
(6). Hence both (5) and (6) survive to homology. If $p$ is odd, we can
still cancel the same arrow; we get an additional map from $w_2$ to
(5), but this just means $w_2$ dies in homology and a rank 1 summand
of $\langle$(5), (6)$\rangle$ survives.

The picture looks similar if we move one step up and to the right. If
$p$ is even, then the wall generator $w_3$ hits only one of the dots
in circle 4, and $w_4$ hits the other one. The maps have coefficient
$\pm 1$, so neither dot in (4) survives to homology.

If $p$ is odd, we have to be careful with signs. Using the sign
algorithm discussed above, the path from $w_3$ to the bottom dot in
(4) gets $p+1$ plus signs, for a net $+$. The path from $w_4$ to the
left dot in (4) gets $p+1$ minus signs, for a net $+$ as well. The
paths from $w_3$ to the left dot in (4) have signs $-$, $+$, $\ldots$,
$-$, so the net sign is $-$. Finally, the paths from $w_4$ to the
bottom dot in (4) also have signs $-$, $+$, $\ldots$, $-$ for a net
$-$. To make it easier for the reader to follow these calculations,
the sign pattern is overlaid in Figure~\ref{7K} for convenience (only
the orange edges are shown, to avoid cluttering the diagram). The
picture looks qualitatively the same for all odd $p$, so checking what
happens when $p = 3$ yields the general pattern.

Now that we know these signs, the homology is easy; the wall
generators contribute a $\Z$ summand, and the floor generators
contribute $\frac{\Z^2}{(1,-1)} \cong \Z$.

Finally, consider the dot marked $f_2$. If $p$ is odd, this dot gets
hit by both $w_5$ and $w_6$, so it does not survive to
homology. However, if $p$ is even, the maps from $w_5$ and $w_6$ to
$f_2$ are zero, so $f_2$ survives to homology.

The rest of the floor generators do not contribute (generators at the
bottoms of the walls are excluded from being floor
generators). Indeed, the process of cancelling these generators using
red arrows from the walls proceeds without needing to consider
non-straight-line paths.

At this point, we can look back at the purple arrows on the walls in
Figure~\ref{6K}; they no longer share tails with any other nonzero
arrows, so we can cancel them just like we did the blue arrows.

\subsubsection{Remaining cases.}  Counting everything up, we have
found $(q-p)(r-p)$ surviving floor generators if $p$ is even and
$(q-p)(r-p)-1$ if $p$ is odd. We obtained this formula assuming $p + 2
\leq q$, $r$, but at this point it is not too difficult to look back
and see what happens in the remaining cases.

First suppose $q = p + 1$. In Figure~\ref{7K}, the missing
``features'' are the dots circled 1, 6, and 3. If $p$ is even, the dot
circled 5 survives to homology, and there are $r-p$ surviving floor
generators in total. This number agrees with our existing formula, so
we do not need to modify it.

If $p$ is odd, (5) does not survive to homology. There are $r - p - 1$
surviving floor generators if $r > p + 1$, and $1$ if $r = p +
1$. Thus the only case where the formula needs modification is when $r
= p + 1$; there the formula would give $0$ surviving generators
instead of the correct number $1$.

Now suppose $q = p$. If $p$ is odd, our formula predicts $-1$ floor
generators, so clearly it needs modification. But each circled floor
generator gets hit by a corresponding wall generator, by an arrow
parallel to the one from $w_1$ to (5). These arrows are nonzero by the
argument of Proposition~\ref{mapprop}, and they can be cancelled one
at a time. Hence no floor generators survive to homology, and we just
need to change the $-1$ to a $0$.

If $p$ is even, the formula requires more drastic modification. The
quantity $(r-p)(q-p)$ is zero, but the generators (2), (5), and $f_2$
survive to homology, for a total of $1$ if $r = q = p$ and $r - p$ if
$r > q - p$.

\subsection{Relative gradings.}  In fact, the floor generators we have
found in homology are all in the lowest possible $\delta$-grading of
$Kh(P(-p,q,r))$, and they comprise the entire homology in this
grading. We can see this by analyzing the relative gradings of various
generators in the complex.

Any time we deloop, the two resulting summands differ by $2$ in the
$q$-grading; the summand without the dot is $2$ steps higher. The
summands have the same homological grading. Also, nonzero components
of the differential must preserve the $q$-grading and increase the
homological grading by one. These two facts will be all we need to
determine the relative gradings of the generators. The second fact
holds before and after any eliminations; we will be sloppy and not
always identify which stage of the elimination process contains the
nonzero component in the differential (it should be clear from
context).

In terms of the grading $\delta = q/2 - h$, two summands from a
delooping differ in $\delta$-grading by 1, and the differential
decreases $\delta$-grading by 1. For simplicity, we will focus on the
$\delta$-grading here.

In the internal columns (right side of Figure~\ref{5K}), there are
only two $\delta$-gradings, say $g$ and $g+1$; the generators on the
left have grading $g+1$, and those on the right have grading $g$. This
observation follows immediately from the two properties stated above;
note that $g$ is the same from column to column because of the
differential components connecting the columns. All the floor
generators (i.e. those surviving after cancellation of the identity
maps in the column) have $\delta$-grading $g$. In particular, this is
true for the homology generators we have found so far.

Now look at generators on the walls (except those in the column where
the walls intersect). All wall generators not on top of the walls have
nonzero components of their differential on floor generators (because
of the zig-zag maps). Hence all these generators must lie in
$\delta$-grading $g+1$. The top nodes on the walls contribute two
generators each; the higher one's differential hits a floor generator,
so the $\delta$-gradings of these two generators are again $g$ and
$g+1$. Note that all the ones in grading $g$ do not survive to
homology (they die when the blue and purple differentials in
Figure~\ref{6K} are cancelled).

Finally, in the column where the walls intersect, all nodes contribute
two generators except the top node which contributes four. In the
two-generator nodes, the $\delta$-gradings are $g+2$ and $g+1$ because
the black arrows are nonzero. In the four-generator node, the gradings
are $g+2$, $g+1$, and $g$, and the generator in grading $g$ gets
cancelled with the blue arrows.

We have now isolated the homology in $\delta$-grading $g$; except in
the special cases, it is $\Z^{(q-p)(r-p)}$ if $p$ is even and
$\Z^{(q-p)(r-p)-1}$ if $p$ is odd. It is also not hard to see what
happens between gradings $g+2$ and $g+1$. Consider the grading-$g+2$
generators in the two-generator nodes of Figure~\ref{6K}. Each has two
nonzero black arrows pointing to generators on the walls. As long as
at least one of these wall generators still exists on each level after
the cancellations we have done, the grading-$g+2$ generators do not
survive to homology. But all of the wall generators still exist except
for one of the two on the very bottom level (depending on which one we
chose to cancel). So we can cancel each grading-$g+2$ generator with a
wall generator using a black arrow.

All the remaining generators are in $\delta$-grading $g+1$, so we can
conclude our complex has no more differentials. We just need to count
generators to determine the rest of the homology; at no point has any
torsion appeared, so the homology is free. The count of generators is
easy if we only care about the $\delta$-graded theory; for the
bigradings, we refer the reader to the appendix.

To count the generators in grading $g+1$, first consider the generic
case ($p + 2 \leq q$, $r$). Note that we started with $qr + q + r + 1$
generators in grading $g$ ($qr$ from the floor generators and $q + r +
1$ from the top layer). There were $(p+1)(r+1) + (p+1)(q+1) - p$ wall
generators in grading $g+1$ and $p+1$ generators in grading $g+2$.  If
$p$ is even, $(q-p)(r-p)$ of the floor generators survive to homology,
so $(qr + q + r + 1) - (q-p)(r-p)$ cancellations occurred between
gradings $g+1$ and $g$. None of the grading-$g+2$ generators survive,
so $p+1$ cancellations occurred between gradings $g+2$ and $g+1$. Thus
the remaining number of generators in grading $g+1$ is
\begin{align*} (p&+1)(r+1) + (p+1)(q+1) - p - ((qr + q + r + 1) -
(q-p)(r-p)) - (p + 1) \\ &= p^2,
\end{align*} after some simplifying. If, instead, $p$ is odd, then one
fewer floor generator survives to homology. Hence there is one more
cancellation between gradings $g+1$ and $g$, and only $p^2 - 1$
generators remain in grading $g+1$.

In summary, in the generic case, if $p$ is even the $\delta$-graded
homology is $\Z_{(g)}^{(q-p)(r-p)} \oplus \Z_{(g+1)}^{p^2}$. If $p$ is
odd, it is $\Z_{(g)}^{(q-p)(r-p)-1} \oplus \Z_{(g+1)}^{p^2-1}$.

For the special cases, if $p$ is odd and $q = r = p+1$, one extra
floor generator survives to homology. Hence we also have an extra
generator in grading $g+1$, and the homology is $\Z_{(g)}^{(q-p)(r-p)}
\oplus \Z_{(g+1)}^{p^2}$ (like in the generic case of $p$ even), which
simplifies to $\Z_{(g)} \oplus \Z_{(g+1)}^{p^2}$.

If $p$ is odd and $q = p$, the answer is similar. Since the generic
formula undershot the homology in grading $g$ by $1$ (by having $-1$
rather than $0$), it also undershot the homology in grading $g+1$ by
$1$. So the homology is $\Z_{(g+1)}^{p^2}$.

If $p$ is even and $q = p$, we understated the homology in grading $g$
by a larger amount: $r - p$ for $r > p$ or $1$ for $r = p$. Hence, if
$r > p$, the homology is $\Z_{(g)}^{r-p} \oplus
\Z_{(g+1)}^{p^2+r-p}$. If $r = p = q$, it is $\Z_{(g)} \oplus
\Z_{(g+1)}^{p^2+1}$

\subsection{Absolute gradings.}\label{absolute} All we need to do to
finish the $\delta$-graded computation is to identify $g$. We will do
this by computing the grading of the grading-$g$ generator in the top
(4-generator) node in Figure~\ref{6K}. No deloopings were performed in
the making of this generator, so its $q$-grading is not affected by
any of the delooping shifts. Its dot-degree is $-2$, and its pattern
of crossing resolutions is $[1\ldots1,0\ldots0,0\ldots0]$ which has
weight $p$. Hence its $q$-grading is $-2 + p + n_+ - 2n_-$. The
numbers $n_+$ and $n_-$ of positive and negative crossings can be
computed from the orientation of the link and Table~\ref{nplusnminus}.

Similarly, the homological grading of this generator is $p -
n_-$. Hence its $\delta$-grading $g$ is $q/2 - h = (-2 - p +
n_+)/2$. We have completed the computation of the $\delta$-graded
Khovanov homology of $P(-p,q,r)$, which we state as a theorem (to be
compared with Theorem~\ref{maintheorem}):

\begin{theorem}\label{maintheorem2} Let $p$, $q$, and $r$ be positive
integers with $p \leq q \leq r$. Let $H_{\delta}$ be the reduced
Khovanov homology of $P(-p,q,r)$ in grading $\delta$ (with basepoint
chosen as in Figure~\ref{3K}).
\begin{itemize}
\item If $p$ is even and $p+2 \leq q$, $r$, then $H_{\frac{-p +
    n_+}{2}} = \Z^{p^2}$ and $H_{\frac{-2-p+n_+}{2}} =
  \Z^{(q-p)(r-p)}$. All other $H_{\delta}$ are 0.
\item If $p$ is odd and $p+2 \leq q$, $r$, then $H_{\frac{-p +
    n_+}{2}} = \Z^{p^2 - 1}$ and $H_{\frac{-2-p+n_+}{2}} =
  \Z^{(q-p)(r-p) - 1}$. All other $H_{\delta}$ are 0.
\item If $p$ is odd and $q = r = p+1$, then $H_{\frac{-p + n_+}{2}} =
  \Z^{p^2}$ and $H_{\frac{-2-p+n_+}{2}} = \Z$. All other $H_{\delta}$
  are 0.
\item If $p$ is even, $q = p$, and $r > p$, then $H_{\frac{-p +
    n_+}{2}} = \Z^{p^2 + r - p}$ and $H_{\frac{-2-p+n_+}{2}} =
  \Z^{r-p}$. All other $H_{\delta}$ are 0.
\item If $p$ is even and $p = q = r$, then $H_{\frac{-p + n_+}{2}} =
  \Z^{p^2+1}$ and $H_{\frac{-2-p+n_+}{2}} = \Z$. All other
  $H_{\delta}$ are 0.
\item If $p$ is odd and $q = p$, then $H_{\frac{-p + n_+}{2}} =
  \Z^{p^2}$. All other $H_{\delta}$ are 0.
\end{itemize} The values of $n_+$ and $n_-$ depend on the orientation
of the link $P(-p,q,r)$ and can be computed from
Table~\ref{nplusnminus}.
\end{theorem}

\section{Appendix.}  Here we wrap up some loose ends, computing the
bigraded homology and making some remarks on the unreduced version.

\subsection{Bigradings.}  We already have a reference point for the
absolute bigradings, computed in Section~\ref{absolute}. We only need
to compute the relative bigradings. Since we know the
$\delta$-gradings, we may focus on the homological grading and use it
to compute the $q$-grading later. This grading is actually much
simpler to compute; every ``forward'' step in the cube (vertically
upwards or horizontally out) increases homological grading by 1,
regardless of whether there are nonzero differentials connecting the
generators. So we can just look through the homology generators we
have found, note where they are in the cube, and deduce their
homological gradings.

For the lower $\delta$-grading $(-2-p+n_+)/2$, all the homology
generators are on the floor of the cube in Figure~\ref{6K}. Hence the
homological gradings correspond to the diagonals of slope $-1$ in the
square of Figure~\ref{7K}. The dot marked $f_2$ has homological
grading $p$ more than the homological grading of the generators we
considered before, which was $p - n_-$. Hence we can start at $f_2$,
with $h$-grading $2p - n_-$, and count dots in subsequent lower
diagonals to get the homological grading of generators in this
$\delta$-grading. This grading can be used along with $\delta$ to
obtain the $q$-grading; in particular, the $q$-grading of $f_2$ is
$2(\delta + h) = -2 - p + n_+ + 4p - 2n_-$, which simplifies to $-2 +
3p + n_+ - 2n_-$.

For convenience, we will write down the formulas more explicitly. Each
of the circled groups of dots in Figure~\ref{7K} potentially
contributes a term to the Khovanov homology, which we will describe by
its Poincar{\'e} polynomial $P_{Kh}$ (over $\Q$, say) because it is
free. We will write the formal variables in this Poincar{\'e}
polynomial by $Q$ and $H$, to avoid confusion with the $q$ we already
have. Let $\phi_{p,q,r}(x) = \sum_{n=0}^{(r-p)+(q-p)-4} c_n x^n$,
where the sequence $c_n$ is defined by the pattern
\begin{align*} (1,&2,3,\ldots,(q-p-2), \\
&(q-p-1),(q-p-1),\ldots,(q-p-1), \\ &(q-p-2),\ldots, 2, 1)
\end{align*} and there are $r-q+1$ instances of $q-p-1$ on the middle
line.

\begin{proposition}\label{floorgens} The dots in Figure~\ref{7K}
represent the following possible contributions to the polynomial
$P_{Kh}(P(-p,q,r))$; whether these summands appear in the formula
depends on whether the dots survive to homology. (Here, and in the
rest of the paper, sums indexed from $0$ to a negative number should
be interpreted as empty.)
\begin{itemize}
\item The dots in circle (1) contribute \[a_1 := Q^{6 + 3p + n_+ -
  2n_-}H^{2p - n_- + 4}\phi_{p,q,r}(Q^2 H)\]
\item The dots in circle (2) contribute \[a_2 := Q^{4 + 3p + n_+ -
  2n_-}H^{2p - n_- + 3}\sum_{n = 0}^{r-p-3}(Q^2 H)^n.\]
\item The dots in circle (3) contribute \[a_3 := Q^{4 + 3p + n_+ -
  2n_-}H^{2p - n_- + 3}\sum_{n = 0}^{q-p-3}(Q^2 H)^n.\]
\item The dots in circle (4) each contribute $Q^{3p + n_+ - 2n_-}H^{2p
  - n_- + 1}$.
\item The dots in circles (5) and (6) each contribute $Q^{2 + 3p + n_+
  - 2n_-}H^{2p - n_- + 2}$.
\item The dot $f_2$ contributes $Q^{-2 + 3p + n_+ - 2n_-}H^{2p -
  n_-}$.
\end{itemize}
\end{proposition}

\begin{figure}
\labellist \small \hair 2pt \pinlabel $w_1$ at 8 344 \pinlabel $w_3$
at 32 352 \pinlabel $w_5$ at 59 361 \pinlabel $w_2$ at 421 344
\pinlabel $w_4$ at 397 350 \pinlabel $w_6$ at 370 358 \pinlabel (1) at
231 416 \pinlabel (2) at 121 375 \pinlabel (3) at 273 385 \pinlabel
(4) at 93 208 \pinlabel (5) at 277 282 \pinlabel $f_2$ at 226 32
\pinlabel $f_1$ at 226 12 \endlabellist \centering
\includegraphics{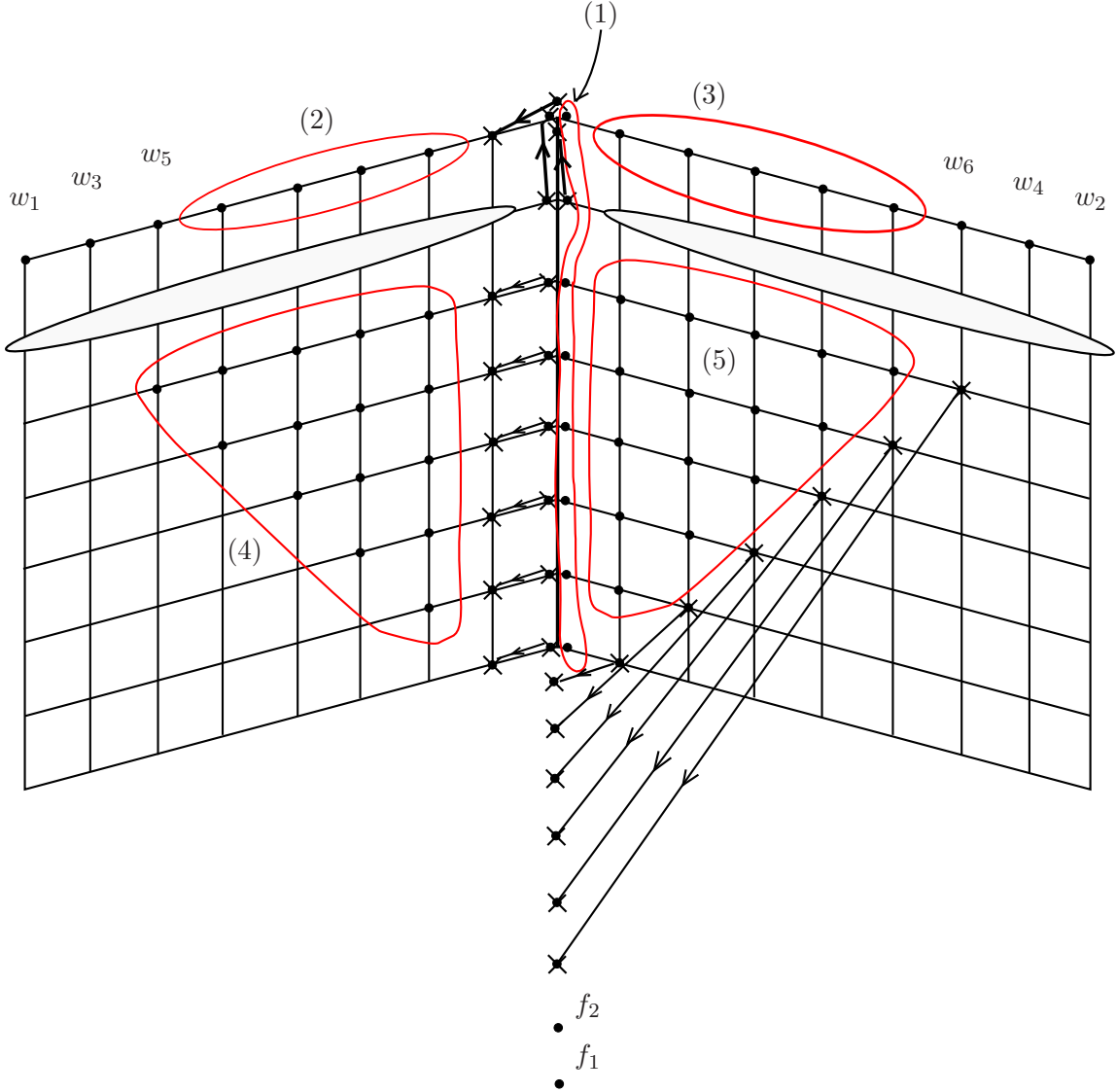}
  \caption{Generators remaining in $\delta$-grading $(-p+n_+)/2$.}
  \label{wallfig}
\end{figure}

Before we go through and total up the contributions for each case of
$P(-p,q,r)$, we will analyze the bigradings in the higher
$\delta$-grading $(-p+n_+)/2$. Figure~\ref{wallfig} shows the relevant
generators, after cancelling all arrows. Dots marked with an x have
been cancelled, while unmarked dots contribute to homology. Their
contributions to $P_{Kh}$ are summarized below; before stating the
formulas, though, we make another definition. Given a positive integer
$k$, let $\psi(k) = \sum_{n=0}^{2k} c_n x^n$, where $c_n$ is defined
by the pattern
\begin{align*} (1,1,2,2,\ldots,&(k/2 - 1), (k/2 - 1),\\ &k/2, k/2,
k/2, \\ &(k/2 - 1), (k/2 - 1), \ldots, 1,1)
\end{align*} if $k$ is even and
\begin{align*} (1,1,2,2,\ldots,&(k-1)/2, (k-1)/2,\\ &(k+1)/2, \\
&(k-1)/2, (k-1)/2, \ldots, 1,1)
\end{align*} if $k$ is odd.

\begin{proposition}\label{wallgens} The dots in Figure~\ref{wallfig}
represent the following contributions to the polynomial
$P_{Kh}(P(-p,q,r))$; for $w_1$ through $w_6$, whether their summands
appear depends on the parity of $p$ (with a few special cases).
\begin{itemize}
\item The dots in circle (1) contribute
\[ b_1 := Q^{- p + n_+ - 2n_-}H^{-n_-} \sum_{n=0}^{p-2} (Q^2 H)^n +
Q^{p + n_+ - 2n_-}H^{p-n_-}. \]
\item The dots in circle (2) contribute \[b_2 := Q^{4 + p + n_+ -
  2n_-}H^{2 + p -n_-} \sum_{n=0}^{p-4} (Q^2 H)^n.\]
\item The dots in circle (3) contribute \[b_3 := Q^{2 + p + n_+ -
  2n_-}H^{1 + p -n_-} \sum_{n=0}^{p-3} (Q^2 H)^n.\]
\item The dots in circle (4) contribute \[b_4 := Q^{6 - p + n_+ -
  2n_-}H^{3 -n_-} \psi_{p-2}(Q^2 H). \]
\item The dots in circle (5) contribute \[b_5 := Q^{4 - p + n_+ -
  2n_-}H^{2 -n_-} \psi_{p-2}(Q^2 H). \]
\item The dots $w_1$ and $w_2$ each contribute $Q^{2 + 3p + n_+ -
  2n_-} H^{1 + 2p - n_-}$.
\item The dots $w_3$ and $w_4$ each contribute $Q^{3p + n_+ - 2n_-}
  H^{2p - n_-}$.
\item The dots $w_5$ and $w_6$ each contribute $Q^{-2 + 3p + n_+ -
  2n_-} H^{-1 + 2p - n_-}$.
\end{itemize}
\end{proposition}

Generically, when $p$ is even, both $w_5$ and $w_6$ contribute to
homology, while both $w_3$ and $w_4$ die. (If $p$ is $2$, however,
$w_5$ coincides with an X in Figure~\ref{wallfig}, so it does not
count.)A rank-one summand of $\{w_1, w_2\}$ survives to homology. When
$p$ is odd, both $\{w_5, w_6\}$ and $\{w_3, w_4\}$ contribute rank-1
summands, and both $w_1$ and $w_2$ die.

For the special cases: when $p$ is odd and $q = r = p+1$, then $w_1$,
$w_3$, and $w_5$ contribute to homology. If $p$ is even, $q = p$, and
$r > p$, then $w_3$ and all the dots in the line to its left in
Figure~\ref{7K} survive, while if $r = p = q$ then only $w_4$
survives. If $p$ is odd and $q = p$, then (besides $w_4$ and $w_6$)
the farthest-left dot on the top line of Figure~\ref{7K}
survives. (Note that we have stopped saying ``rank-1 summand,'' for
convenience, choosing instead to pick arbitrarily which generator
survives in some cases.)

We are ready to state the formula for the bigraded reduced Khovanov
homology of $P(-p,q,r)$. There will be several cases, and we will use
the polynomials $a_i$ and $b_i$ from Proposition~\ref{floorgens} and
Proposition~\ref{wallgens} respectively.

\begin{theorem}\label{maintheorembig} Let $p$, $q$, and $r$ be
positive integers with $p \leq q \leq r$. The reduced (even) Khovanov
homology of $P(-p,q,r)$ in grading $\delta$ (with basepoint chosen as
in Figure~\ref{3K}) is free over $\Z$. Let $P_{Kh}$ be its
Poincar{\'e} polynomial over $\Q$; then $P_{Kh}$ is given by the
following formulas when $p \geq 3$. (Sums from 0 to a negative number
are taken to be empty, as before.)

When $p = 2$, note that each case of ``$p$ even'' in the formula has a
term of the form $2Q^{-2 + 3p + n_+ - 2n_-} H^{-1 + 2p - n_-}$. The
coefficient $2$ here should be replaced with a $1$ when $p = 2$,
because the dot $w_5$ is missing from the count of generators.
\begin{itemize}

\item If $p$ is even and $p+2 \leq q$, $r$, then
\begin{align*} P_{Kh} = \sum_{n=1}^3 a_i + \sum_{n=1}^5 b_i &+ 2Q^{2 +
3p + n_+ - 2n_-}H^{2p - n_- + 2} + Q^{-2 + 3p + n_+ - 2n_-}H^{2p -
    n_-} \\ &+ Q^{2 + 3p + n_+ - 2n_-} H^{1 + 2p - n_-} + 2Q^{-2 + 3p
    + n_+ - 2n_-} H^{-1 + 2p - n_-}.
\end{align*}

\item If $p$ is odd and $p+2 \leq q$, $r$, then
\begin{align*} P_{Kh} = \sum_{n=1}^3 a_i + \sum_{n=1}^5 b_i &+ Q^{2 +
3p + n_+ - 2n_-}H^{2p - n_- + 2} + Q^{3p + n_+ - 2n_-}H^{2p - n_- + 1}
  \\ &+ Q^{3p + n_+ - 2n_-} H^{2p - n_-} + Q^{-2 + 3p + n_+ - 2n_-}
  H^{-1 + 2p - n_-}.
\end{align*}

\item If $p$ is even, $q = p+1$, and $r > q$, then
\begin{align*} P_{Kh} = a_2 + \sum_{n=1}^5 b_i &+ Q^{2 + 3p + n_+ -
2n_-}H^{2p - n_- + 2} + Q^{-2 + 3p + n_+ - 2n_-}H^{2p - n_-} \\ &+
  Q^{2 + 3p + n_+ - 2n_-} H^{1 + 2p - n_-} + 2Q^{-2 + 3p + n_+ - 2n_-}
  H^{-1 + 2p - n_-}.
\end{align*}

\item If $p$ is odd, $q = p + 1$, and $r > q$, then
\begin{align*} P_{Kh} = a_2 + \sum_{n=1}^5 b_i &+ Q^{3p + n_+ -
2n_-}H^{2p - n_- + 1} \\ &+ Q^{3p + n_+ - 2n_-} H^{2p - n_-} + Q^{-2 +
    3p + n_+ - 2n_-} H^{-1 + 2p - n_-}.
\end{align*}

\item If $p$ is even and $r = q = p+1$, then
\begin{align*} P_{Kh} = \sum_{n=1}^5 b_i &+ Q^{-2 + 3p + n_+ -
2n_-}H^{2p - n_-} \\ &+ Q^{2 + 3p + n_+ - 2n_-} H^{1 + 2p - n_-} +
  2Q^{-2 + 3p + n_+ - 2n_-} H^{-1 + 2p - n_-}.
\end{align*}

\item If $p$ is odd and $r = q = p+1$, then
\begin{align*} P_{Kh} = &\sum_{n=1}^5 b_i + Q^{3p + n_+ - 2n_-}H^{2p -
n_- + 1} \\ &+ Q^{3p + n_+ - 2n_-} H^{2p - n_-} + Q^{-2 + 3p + n_+ -
    2n_-} H^{-1 + 2p + n_-} \\ &+ Q^{2 + 3p + n_+ - 2n_-} H^{1 + 2p -
    n_-}.
\end{align*}

\item If $p$ is even, $q = p$, and $r > p + 1$, then
\begin{align*} P_{Kh} = a_2 + \sum_{n=1}^5 b_i &+ Q^{-2 + 3p + n_+ -
2n_-}H^{2p - n_-} + Q^{2 + 3p + n_+ - 2n_-}H^{2p - n_- + 2}\\ &+ Q^{2
    + 3p + n_+ - 2n_-} H^{1 + 2p - n_-}\sum_{n=0}^{r-p-1} (Q^2 H)^n
  \\ &+ Q^{3p + n_+ - 2n_-} H^{2p - n_-} \\ &+ 2Q^{-2 + 3p + n_+ -
    2n_-} H^{-1 + 2p - n_-}.
\end{align*}

\item If $p$ is even, $q = p$, and $r = p + 1$, then
\begin{align*} P_{Kh} = \sum_{n=1}^5 b_i &+ Q^{-2 + 3p + n_+ -
2n_-}H^{2p - n_-} \\ &+ Q^{2 + 3p + n_+ - 2n_-} H^{1 + 2p - n_-} +
  Q^{3p + n_+ - 2n_-} H^{2p - n_-}\\ &+ 2Q^{-2 + 3p + n_+ - 2n_-}
  H^{-1 + 2p - n_-}.
\end{align*}

\item If $p$ is even and $p = q = r$, then
\begin{align*} P_{Kh} = \sum_{n=1}^5 b_i &+ Q^{-2 + 3p + n_+ -
2n_-}H^{2p - n_-} \\ &+ 2Q^{3p + n_+ - 2n_-} H^{2p - n_-} + 2Q^{-2 +
    3p + n_+ - 2n_-} H^{-1 + 2p - n_-}.
\end{align*}

\item If $p$ is odd and $q = p$, then
\begin{align*} P_{Kh} = &\sum_{n=1}^5 b_i \\ &+ Q^{-2 + 3p + n_+ -
2n_-} H^{-1 + 2p - n_-} + Q^{3p + n_+ - 2n_-} H^{2p - n_-} \\ &+ Q^{3p
    + n_+ - 2n_- + 2(r-p)} H^{2p - n_- + (r-p)}.
\end{align*}

\end{itemize} The values of $n_+$ and $n_-$ depend on the orientation
of the link $P(-p,q,r)$ and can be computed from
Table~\ref{nplusnminus}.
\end{theorem}
\begin{proof} The proof consists of carefully looking at
Figure~\ref{7K} and Figure~\ref{wallfig} and counting up the
contributing generators. To save space, and since we have already
indicated how to do this counting, we will omit a more detailed proof
here.
\end{proof}

\subsection{Unreduced homology.}  Although we focused on the reduced
Khovanov homology in this paper, it would not be too difficult to use
this computation to obtain the unreduced homology. The Lee spectral
sequences on reduced and unreduced homology, together with the exact
sequence relating two copies of the reduced homology with the
unreduced homology, give lots of information about the unreduced
homology given the reduced version. In the examples the author
computed, this information was enough to determine the unreduced
homology. As described in \cite{Manion}, the unreduced homology
consists only of knight's-move pairs and exceptional pairs (see
Bar-Natan \cite{OnKhovCat}); while \cite{Manion} works only over $\Q$,
the integral unreduced homology ends up having copies of $\Z_2$ as
expected in the knight's-move pairs and is free otherwise. We will
forego a more rigorous discussion because it would lengthen the paper
without necessarily adding more insight.

\bibliographystyle{plain} \bibliography{biblio}
\end{document}